\documentclass[10pt]{amsart}
\usepackage[utf8]{inputenc}

\usepackage{amssymb,amsthm,amsmath}
\usepackage{enumerate}
\usepackage{graphicx,color}
\usepackage[hidelinks]{hyperref}

\newcommand{\dd}{\mathrm{d}}
\newcommand{\E}{\mathbb{E}}
\newcommand{\1}{\textbf{1}}
\newcommand{\R}{\mathbb{R}}

\newcommand{\p}[1]{\mathbb{P}\left( #1 \right)}

\DeclareMathOperator{\Var}{Var}

\usepackage[paper=a4paper, left=1.5in, right=1.5in, top=1.4in, bottom=1.4in]{geometry}
\newtheorem{theorem}{Theorem}
\newtheorem{lemma}[theorem]{Lemma}

\newtheorem{proposition}[theorem]{Proposition}
\theoremstyle{remark}
\newtheorem{remark}[theorem]{Remark}

\theoremstyle{definition}

\title{Tail bounds for sums of independent two-sided exponential random variables}

\author{Jiawei Li}
\address{(J.L.) Carnegie Mellon University; Pittsburgh, PA 15213, USA.}
\email{jiaweil4@andrew.cmu.edu}

\author{Tomasz Tkocz}
\address{(T.T.) Carnegie Mellon University; Pittsburgh, PA 15213, USA.}
\email{ttkocz@math.cmu.edu}

\thanks{TT's research supported in part by NSF grant DMS-1955175.}

\date{\today}

\begin{document}

\begin{abstract}
We establish upper and lower bounds with matching leading terms for tails of weighted sums of two-sided exponential random variables. This extends Janson's recent results for one-sided exponentials.
\end{abstract}

\maketitle

\bigskip

\begin{footnotesize}

\end{footnotesize}

\bigskip

\section{Introduction}

Concentration inequalities establish conditions under which random variables are \emph{close} to their \emph{typical values} (such as the expectation or median) and provide quantitative probabilistic bounds. Their significance cannot be overestimated, both across probability theory and in applications in related areas (see \cite{BDR, BLM}). Particularly, such inequalities often concern sums of independent random variables.

Let $X_1, \ldots, X_n$ be independent exponential random variables, each with mean $1$. Consider their weighted sum $S = \sum_{i=1}^n a_iX_i$ with some positive weights $a_1, \ldots, a_n$. Janson in \cite{Jan} showed the following concentration inequality: for every $t > 1$,
\begin{equation}\label{eq:J}
\frac{1}{2e\alpha}\exp\big(-\alpha(t-1)\big) \leq \p{S \geq t \E S} \leq \frac{1}{t}\exp\big(-\alpha(t-1-\log t)\big),
\end{equation}
where $\alpha = \frac{\E S}{\max_{i \leq n} a_i}$ (in fact, he derived \eqref{eq:J} from its analogue for the geometric distribution). Note that as $t \to \infty$, the lower and upper bounds are of the same order $e^{-\alpha t + o(t)}$. Moreover, $e^{-\alpha t} = \p{X_1 > t\frac{\E S}{\max_{i \leq n} a_i}}$. In words, the asymptotic behaviour of the tail of the sum $S$ is the same as that of one summand carrying the largest weight. 

The goal of this short note is to exhibit that the same behaviour holds for sums of two-sided exponentials (Laplace). Our main result reads as follows.

\begin{theorem}\label{thm:sym-exp}
Let $X_1, \ldots, X_n$ be independent standard two-sided exponential random variables (i.e. with density $\frac{1}{2}e^{-|x|}$, $x \in \R$). Let $S = \sum_{i=1}^n a_iX_i$ with $a_1,\ldots,a_n$ positive. For every $t > 1$,
\begin{equation}\label{eq:sym-exp}
\frac{1}{57}\frac{1}{\sqrt{\alpha t}}\exp\big(-\alpha t\big) \leq \p{S > t\sqrt{\Var(S)}} \leq \exp\left(-\frac{\alpha^2}{2}h\left(\frac{2t}{\alpha}\right)\right), 
\end{equation}
where $\alpha = \frac{\sqrt{\Var(S)}}{\max_{i \leq n} a_i}  = \frac{\sqrt{2\sum_{i=1}^n a_i^2}}{\max_{i \leq n} a_i}$, $h(u) = \sqrt{1+u^2}-1-\log\frac{1+\sqrt{1+u^2}}{2}$, $u > 0$.
\end{theorem}

In \eqref{eq:sym-exp}, as $t \to \infty$, the lower and the upper bounds are of the same order, $e^{-\alpha t + o(t)}$ (plainly, $h(u) = u + o(u)$).

Our proof of Theorem \ref{thm:sym-exp} presented in Section \ref{sec:proof} is based on an observation that two-sided exponentials are Gaussian mixtures, allowing to leverage \eqref{eq:J} (this idea has recently found numerous uses in convex geometry, see \cite{Esk, ENT, NT}). In Section \ref{sec:gen}, we provide further generalisations of Janson's inequality \eqref{eq:J} to certain nonnegative distributions, which also allows to extend Theorem \ref{thm:sym-exp} to a more general framework. We finish in Section \ref{sec:rem} with several remarks (for instance, we deduce from \eqref{eq:sym-exp} a formula for moments of $S$).

\subsection*{Acknowledgements}
We are indebted to an anonymous referee for many valuable comments, leading in particular to the remarks in Sections \ref{sec:S-ineq} and \ref{sec:heavy-tailed}.

\section{Proof of Theorem \ref{thm:sym-exp}}\label{sec:proof}

For the upper bound, we begin with a standard Chernoff-type calculation. Denote $\sigma = \sqrt{\Var(S)} = \sqrt{2\sum a_i^2}$. For $\theta > 0$, we have
\[
\p{S \geq t \sigma} \leq e^{-\theta t \sigma}\E e^{\theta S}
\]
and
\[
\E e^{\theta S} = \prod \E e^{\theta a_iX_i} = \prod \frac{1}{1-\theta^2a_i^2} = \exp\left\{-\sum \log(1-\theta^2a_i^2)\right\},
\]
for $\theta < \frac{1}{a_*}$, $a_* = \max_{i \leq n} a_i$. By convexity,
\[
-\sum \log(1-\theta^2a_i^2) \leq -\sum \frac{a_i^2}{a_*^2}\log(1-\theta^2a_*^2),
\]
so changing $\theta$ to $\theta/a_*$, for every $0 < \theta < 1$, we have
\[
\p{S \geq t \sigma} \leq \exp\left\{-\theta t\alpha -\frac{\alpha^2}{2}\log(1-\theta^2)\right\} = \exp\left\{-\frac{\alpha^2}{2}\left(\frac{2t}{\alpha}\theta + \log(1-\theta^2\right)\right\},
\]
where $\alpha = \frac{\sigma}{a_*}$.
Optimising over $\theta$ and using
\[
\sup_{\theta \in (0,1)} \Big(\theta u + \log(1-\theta^2)\Big) = \sqrt{1+u^2}-1-\log\frac{1+\sqrt{1+u^2}}{2}, \qquad u > 0
\]
gives the upper bound in \eqref{eq:sym-exp} and thus finishes the argument.

For the lower bound, we shall use that a standard two-sided exponential random variable with density $\frac{1}{2}e^{-|x|}$, $x \in \R$, has the same distribution as $\sqrt{2Y}G$, where $Y$ is an exponential random variable with mean $1$ and $G$ is a standard Gaussian random variable independent of $Y$ (this follows for instance by checking that the characteristic functions are the same; see also a remark following Lemma 23 in \cite{ENT}). This and the fact that sums of independent Gaussians are Gaussian justify the following claim, central to our argument.

\begin{proposition}
The sum $S = \sum_{i=1}^n a_iX_i$ has the same distribution as
$
(2\sum_{i=1}^n a_i^2Y_i)^{1/2} G
$
with $Y_1,\ldots, Y_n$ being independent mean $1$ exponential random variables, independent of the standard Gaussian $G$. \label{prop:exp-decomp}
\end{proposition}

Recall $\alpha = \frac{\sigma}{\max a_i}$. Fix $t > 1$. By Proposition \ref{prop:exp-decomp}, for $\theta > 0$, we have
\begin{align*}
\p{S \geq t\sigma} = \p{\sqrt{2\sum a_i^2Y_i}G \geq t\sigma} &\geq \p{\sqrt{2\sum a_i^2Y_i} \geq \sqrt{\theta t\sigma^2}, \ G \geq \sqrt{\theta^{-1}t}} \\
&= \p{\sum a_i^2Y_i \geq \frac{1}{2}\theta t\sigma^2}\p{G \geq \sqrt{\theta^{-1}t}}.
\end{align*}

\emph{Case 1.} $t \geq \alpha$. With hindsight, choose $\theta = \frac{1}{\alpha}$. Applying \eqref{eq:J} to the first term yields
\[
\p{\sum a_i^2Y_i \geq \frac{1}{2}\theta t\sigma^2} = \p{\sum a_i^2Y_i \geq \frac{t}{\alpha}\sum a_i^2} \geq \frac{1}{e\alpha^2}\exp\left\{-\frac{\alpha^2}{2}\left(\frac{t}{\alpha}-1\right)\right\}.
\]
For the second term we use a standard bound on the Gaussian tail,
\begin{align}
\p{G > u} &\geq \frac{1}{\sqrt{2\pi}}\frac{u}{u^2+1}e^{-u^2/2}, \qquad\   u > 0, \notag\\
&\geq \frac{1}{2\sqrt{2\pi}}\frac{1}{u}e^{-u^2/2}, \qquad\qquad u \geq 1.\label{eq:Gauss-tail}
\end{align}
and as $\theta^{-1}t= \alpha t \geq \sqrt{2}$, (\ref{eq:Gauss-tail}) applies in our case.
Combining the above estimates gives
\begin{align*}
\p{S \geq t\sigma} \geq \frac{\exp(\alpha^2/2)}{2\sqrt{2\pi}e\alpha^2}\frac{1}{\sqrt{\alpha t}}\exp\big(-\alpha t\big) \geq \frac{1}{4\sqrt{2\pi}}\frac{1}{\sqrt{\alpha t}}\exp\big(-\alpha t\big),
\end{align*}
where in the last inequality we use that $\inf_{x > 1}\frac{1}{x}e^{x/2} = \frac{e}{2}$.

\emph{Case 2.} $t \leq \alpha$. With hindsight, choose $\theta = \frac{1}{t}$. Then
\[
\p{\sum a_i^2Y_i \geq \frac{1}{2}\theta t\sigma^2} = \p{\sum a_i^2Y_i \geq \sum a_i^2}.
\]
To further lower-bound the last expresion, we use a standard Paley-Zygmund type inequality (see, e.g. Lemma 3.2 in \cite{Ole}).

\begin{lemma}\label{lm:PZ}
Let $Z_1, \ldots, Z_n$ be independent mean $0$ random variables such that $\E Z_i^4 \leq C(\E Z_i^2)^2$ for all $1\leq i \leq n$ for some constant $C \geq 1$. Then for $Z = Z_1+\dots+Z_n$,
\[
\p{Z \geq 0} \geq \frac{1}{16^{1/3}\max\{C,3\}}.
\]
\end{lemma}
\begin{proof}
We can assume that $\p{Z = 0} < 1$. Since $Z$ has mean $0$,
\[
\E|Z| = 2\E Z\1_{Z \geq 0} \leq 2(\E Z^4)^{1/4}\p{Z \geq 0}^{3/4}.
\]
Moreover, by H\"older's inequality, $\E|Z| \geq \frac{(\E Z^2)^{3/2}}{(\E Z^4)^{1/2}}$, so
\[
\p{Z \geq 0} \geq 16^{-1/3}\frac{(\E Z^2)^2}{\E Z^4}.
\]
Using independence, $\E Z_i = 0$ and the assumption $\E Z_i^4 \leq C(\E Z_i^2)^2$, we have
\begin{align*}
\E Z^4 = \sum_{i=1}^n \E Z_i^4 + 6\sum_{i < j}\E Z_i^2\E Z_j^2 &\leq \max\{C,3\}\left(\sum_{i=1}^n (\E Z_i^2)^2 + 2\sum_{i < j}\E Z_i^2\E Z_j^2 \right) \\
&= \max\{C,3\}(\E Z^2)^2.
\end{align*}
\end{proof}

Take $Z_i = a_i(Y_i - 1)$. We have, $\E(Y_i-1)^2 = 1$, $\E(X_i-\gamma)^4 = 9$. Thus we can apply Lemma \ref{lm:PZ} with $C = 9$ and obtain
\begin{equation}\label{eq:S>ES}
\p{\sum a_i^2Y_i \geq \sum a_i^2} \geq \frac{1}{9\cdot 16^{1/3}}.
\end{equation}
By \eqref{eq:Gauss-tail},
\[
\p{G \geq \sqrt{\theta^{-1}t}} = \p{G \geq t} \geq \frac{1}{2\sqrt{2\pi}}\frac{1}{t}e^{-t^2/2} \geq \frac{1}{2\sqrt{2\pi}}\frac{1}{\sqrt{\alpha t}}e^{-\alpha t/2},
\]
where in the last inequality we use that in this case $t \leq \sqrt{\alpha t}$. Moreover, since $\alpha t \geq \sqrt{2}$, $e^{-\alpha t/2}\geq e^{1/\sqrt{2}}e^{-\alpha t}$. Thus,
\[
\p{S \geq t\sigma} \geq \frac{e^{1/\sqrt{2}}}{18\cdot 16^{1/3}\sqrt{2\pi}}\frac{1}{\sqrt{\alpha t}}\exp\big(-\alpha t\big) > \frac{1}{57}\frac{1}{\sqrt{\alpha t}}\exp\big(-\alpha t\big).
\]
Combining Case 1 and 2 finishes the proof of the lower bound in \eqref{eq:sym-exp} and thus the proof proof of Theorem \ref{thm:sym-exp} is complete.\hfill$\square$

\section{Generalisations}\label{sec:gen}

In this section, we provide general tail bounds for weighted sums of nonnegative random variables which for certain distributions allow to capture the same bahaviour as featured in Janson's inequality \eqref{eq:J}, viz. asymptotically the sum has the same tail as the summand carrying the largest weight.

\begin{theorem}\label{thm:gene}
Let $X_1, \ldots, X_n$ be i.i.d. nonnegative random variables, $\mu = \E X_1$.
Let $S = \sum_{i=1}^n a_iX_i$ with $a_1,\ldots,a_n$ positive. For every $t > 1$,
\begin{equation}\label{eq:gene}
\p{S \geq \E S} r((t-1)\alpha\mu) \leq \p{S > t\E S} \leq \exp\left\{-\alpha I(\mu t)\right\},
\end{equation}
where $\alpha = \frac{\sum_{i=1}^n a_i}{\max_{i \leq n} a_i}$, for $v>0$,
\begin{equation}\label{eq:r}
r(v) = \inf_{u > 0} \frac{\p{X_1 > u+v}}{\p{X_1 > u}}
\end{equation}
and for $t > 0$,
\begin{equation}\label{eq:L}
I(t) = \sup_{\theta > 0}\left(t \theta - \log\E e^{\theta X_1}\right).
\end{equation}
\end{theorem}

Before presenting the proof, we look at the example of the exponential and gamma distribution.

\subsection{Examples}

When the $X_i$ are exponential rate $1$ random variables, $I(t) = t-1-\log t$, $r(v) = e^{-v}$, $\p{S \geq \E S} \geq \frac{1}{9\cdot 16^{1/3}}$ (see \eqref{eq:S>ES}) and we obtain
\[
\frac{1}{9\cdot 16^{1/3}}e^{-\alpha (t-1)} \leq \p{S > t\E S} \leq e^{-\alpha(t-1-\log t)}.
\]
Comparing with \eqref{eq:J}, the extra factor $\frac1t$ in the upper bound was obtained in \cite{Jan} through rather delicate computations for the moment generating function specific for the exponential distribution. Since $\alpha \geq 1$, our lower bound up to a universal constant recovers the one from \eqref{eq:J} (improves on it as long as $\alpha > 9\cdot 16^{1/3}/(2e)$ and is worse otherwise). Along the same lines, for the gamma distribution with parameter $\gamma > 0$ (i.e. with density $\Gamma(\gamma)^{-1}x^{\gamma - 1}e^{-x}$, $x > 0$), we have $\mu = \gamma$, $I(t\mu) = \gamma(t-1-\log t)$ and with some extra work,
\[
r_\gamma(v) = \begin{cases} \frac{1}{2\Gamma(\gamma)}\min\{v^{\gamma-1},1\}e^{-v}, & 0 < \gamma < 1, \\ e^{-v}, & \gamma \geq 1. \end{cases}
\]
Moreover, via Lemma \ref{lm:PZ}, $\p{S \geq \E S} > \frac{1}{3\cdot 16^{1/3}(1+2\gamma^{-1})}$. Then \eqref{eq:gene} yields 
\begin{equation}\label{eq:gamma}
\frac{1}{3\cdot 16^{1/3}(1+2\gamma^{-1})} r_\gamma\big(\alpha\gamma(t-1)\big) \leq \p{S > t\E S} \leq \exp\big(-\alpha\gamma(t-1-\log t)\big).
\end{equation}
In particular,
$\p{S > t \E S} = \exp\{-\alpha \gamma t + o(t)\}$ as $t \to \infty$. It would perhaps be interesting to find a larger class of distributions for which the upper and lower bounds from \eqref{eq:gene} are asymptotically tight.  For more precise results involving the variance of $S$ for weighted sums of independent Gamma random variables (not necessarily with the same parameter), we refer to Theorem 2.57 in \cite{BDR}.

\subsection{Proof of Theorem \ref{thm:gene}: the upper bound}

For the log-moment generating function $\psi\colon\R\to (-\infty,\infty]$,
\[
\psi(u) = \log \E e^{u X_1}, \qquad u \in \R,
\]
we have $\psi(0) = 0$, $\psi$ is convex (by H\"older's inequality). Thus, by the monotonicity of slopes of convex functions,
\begin{equation}\label{eq:key}
\R \ni u \mapsto \frac{\psi(u)}{u} \ \ \text{is nondecreasing}.
\end{equation}
This is what Janson's proof specified to the case of exponentials relies on.
We turn to estimating the tails (using of course Chernoff-type bounds). 
Fix $t > 1$. For $\theta > 0$, we have
\begin{align*}
\p{S \geq t \E S} = \p{e^{\theta S} \geq e^{\theta t \E S}} \leq e^{-\theta t \E S} \E e^{\theta S} &= e^{-\theta t \E S} \prod_{i=1}^n \E e^{\theta a_iX_i}\\
&= \exp\left\{-\theta t \E S + \sum_{i=1}^n \psi(\theta a_i)\right\}.
\end{align*}
Let $a_* = \max_{i \leq n} a_i$. Thanks to \eqref{eq:key},
\[
\sum_{i=1}^n \psi(\theta a_i) = \sum_{i=1}^n (\theta a_i)\frac{\psi(\theta a_i)}{\theta a_i} \leq \sum_{i=1}^n (\theta a_i)\frac{\psi(\theta a_*)}{\theta a_*} = \frac{\sum_{i=1}^n a_i}{a_*}\psi(\theta a_*) = \alpha\psi(\theta a_*),
\]
where we set $\alpha = \frac{\sum_{i=1}^n a_i}{a_*}$. Note $\E S = \mu\sum a_i = \mu\alpha a_*$. We obtain,
\[
\p{S \geq t \E S} \leq \exp\left\{-\theta t \E S + \alpha\psi(\theta a_*)\right\} = \exp\left\{-\alpha\left(t\mu \theta a_*  - \psi(\theta a_*)\right)\right\},
\]
so optimising over $\theta$ gives the upper bound of \eqref{eq:gene}.\hfill$\square$

\subsection{Proof of Theorem \ref{thm:gene}: the lower bound}

We follow a general idea from \cite{Jan}. The whole argument is based on the following simple lemma.

\begin{lemma}\label{lm:decay-sum}
Suppose $X$ and $Y$ are independent random variables and $Y$ is such that $\p{Y \geq u+v} \geq r(v)\p{Y \geq u}$ for all $u \in \R$ and $v > 0$, for some function $r(v)$. Then $\p{X+Y \geq u+v} \geq r(v) \p{X+Y \geq u}$ for all $u \in \R$ and $v> 0$.
\end{lemma}
\begin{proof}
By independence, conditioning on $X$, we get
\begin{align*}
\p{X+Y \geq u+v} = \E_X\mathbb{P}_Y(Y \geq u-X + v) &\geq r(v)\E_X\mathbb{P}_Y(Y \geq u-X) \\
&= r(v)\p{X+Y \geq u}.
\end{align*}
\end{proof}

Let $S = \sum_{i=1}^n a_iX_i$ be the weighted sum of i.i.d. random variables and without loss of generality let us assume $a_1 = \max_{i \leq n} a_i$. Fix $t > 1$. We write $S = S' + a_1X_1$, with $S'=\sum_{i=2}^n a_i X_i$. Note that the definition of function $r$ from \eqref{eq:r} remains unchanged if the infimum is taken over all $u \in \R$ (since $X_1$ is nonnegative). Thus  Lemma \ref{lm:decay-sum} gives
\begin{align*}
\p{S \geq t\E S} &= \p{S \geq \E S + (t-1)\E S} \geq r\left((t-1)\frac{\E S}{a_1}\right)\p{S \geq \E S},
\end{align*}
as desired.\hfill$\square$

\section{Further remarks}\label{sec:rem}

\subsection{Moments}

The upper bound from \eqref{eq:sym-exp} allows us to recover precise estimates for moments (a special case of Gluskin and Kwapie\'n results from \cite{GK}), with a straightforward proof. Here and throughout, $\|a\|_p = (\sum_{i=1}^n |a_i|^p)^{1/p}$ denotes the $p$-norm of a sequence $a = (a_1,\ldots,a_n)$, $p > 0$, and $\|a\|_\infty = \max_{i \leq n} |a_i|$.
\begin{theorem}[Gluskin and Kwapie\'n, \cite{GK}]\label{thm:mom}
Under the assumptions of Theorem \ref{thm:sym-exp}, for every $p \geq 2$,
\begin{equation}\label{eq:mom}
\frac{\sqrt{2e}}{\sqrt{2e}+1}\big(p\|a\|_\infty + \sqrt{p}\|a\|_2\big) \leq \left(\E\left|S\right|^p\right)^{1/p} \leq 4\sqrt{2}\big(p\|a\|_\infty + \sqrt{p}\|a\|_2\big).
\end{equation}
\end{theorem}
\begin{proof}
For the upper bound, letting $\tilde S = \frac{S}{\sqrt{\Var(S)}}$ and using \eqref{eq:sym-exp}, we get
\begin{align*}
\E|\tilde S|^p = \int_0^\infty pt^{p-1}\p{|\tilde S|>t} \dd t \leq \int_0^1 pt^{p-1} \dd t + 2\int_1^\infty pt^{p-1}\exp\left(-\frac{\alpha^2}{2}h\left(\frac{2t}{\alpha}\right)\right)\dd t.
\end{align*}
We check that as $u$ increases, $h(u)$ behaves first quadratically, then linearly. More precisely,
\begin{equation}\label{eq:h-low-bd}
h(u) \geq \frac{1}{5}u^2, \quad u \in (0,\sqrt{2}), \qquad h(u) \geq \frac{1}{4}u, \quad u \in (\sqrt{2},\infty).
\end{equation}
Thus the second integral $\int_1^\infty \ldots \dd t$ can be upper bounded by (recall that $\Var(S) = 2\|a\|_2^2$, $\frac{\alpha}{\sqrt{2}} = \frac{\|a\|_2}{\|a\|_\infty} > 1$),
\begin{align*}
&\int_1^{\alpha/\sqrt{2}} pt^{p-1}\exp\left(-\frac{\alpha^2}{2}\frac{1}{5}\left(\frac{2t}{\alpha}\right)^2\right)\dd t + \int_{\alpha/\sqrt{2}}^\infty pt^{p-1}\exp\left(-\frac{\alpha^2}{2}\frac{1}{4}\frac{2t}{\alpha}\right)\dd t \\
&\leq \int_0^\infty pt^{p-1}\exp\left(-\frac{2}{5}t^2\right)\dd t + \int_0^\infty pt^{p-1}\exp\left(-\frac{1}{4}\alpha t\right)\dd t \\
&= \left(\frac{5}{2}\right)^{p/2}\Gamma\left(\frac{p}{2}+1\right) + \left(\frac{4}{\alpha}\right)^p\Gamma(p+1).
\end{align*}
Using $\Gamma(x+1) \leq x^x$, $x \geq 1$, yields
\begin{align*}
\left(\E|S|^p\right)^{1/p}=\sqrt{2}\|a\|_2\left(\E|\tilde S|^p\right)^{1/p} &\leq \sqrt{2}\|a\|_2\left(1 + 2\left(\frac{5p}{4}\right)^{p/2} + 2\left(\frac{4p}{\alpha}\right)^p\right)^{1/p} \\
&\leq 4\sqrt{2}(p\|a\|_\infty + \sqrt{p}\|a\|_2).
\end{align*}

For the lower bound, suppose $a_1 = \|a\|_\infty$. Then, by independence and Jensen's inequality,
\[
\E|S|^p \geq \E\big|a_1X_1 + \E(a_2X_2+\dots+a_nX_n)\big|^p = a_1^p\E|X_1|^p = a_1^p\Gamma(p+1).
\]
Using $\Gamma(x+1)^{1/x} \geq x/e$, $x > 0$ (Stirling's formula, \cite{Ja}), this gives
\[
(\E|S|^p)^{1/p} \geq \frac{p}{e}\|a\|_\infty.
\]
On the other hand, by Proposition \ref{prop:exp-decomp}, and Jensen's inequality,
\[
\E|S|^p = \E\left(2\sum a_i^2Y_i\right)^{p/2}\E|G|^p \geq \left(2\sum a_i^2\right)^{p/2}\E|G|^p.
\]
Using $\E|G|^p \geq (p/e)^{p/2}$, $p \geq 1$ (again, by e.g. Stirling's approximation), we obtain
\[
(\E|S|^p)^{1/p} \geq \sqrt{\frac{2}{e}}\sqrt{p}\|a\|_2.
\]
Combining gives 
\[
(\E|S|^p)^{1/p} \geq \max\left\{\frac{1}{e}p\|a\|_\infty, \sqrt{\frac{2}{e}}\sqrt{p}\|a\|_2\right\}  \geq \frac{\sqrt{2e}}{\sqrt{2e}+1}\left(p\|a\|_\infty + \sqrt{p}\|a\|_2\right),
\]
which finishes the proof.
\end{proof}

\begin{remark}\label{rem:GK}
Using Markov and Payley-Zygmund type inequalities, it is possible to recover two-sided tail bounds from moment estimates (like \eqref{eq:mom}), but incurring loss of (universal) constants in the exponents, as it is done in e.g. \cite{GK}, or \cite{HMS}.
\end{remark}

\subsection{Upper bounds on upper tails from S-inequalities}\label{sec:S-ineq}
Let $S$ be as in \eqref{eq:J}. The upper bound in \eqref{eq:J} for $t=1$ is trivial, whereas as a result of  Lemma \ref{lm:PZ}, viz. \eqref{eq:S>ES}, we obtain $\p{S \geq \E S} \in (\frac{1}{24}, \frac{23}{24})$, where the upper bound $\frac{23}{24}$ is obtained by applying Lemma \ref{lm:PZ} to $-Z$. Letting $a > 0$ be such that $\p{S \geq \E S} = \p{X_1 \geq a} = e^{-a}$, by the S-inequality for the two-sided product exponential measure and the set $\{x \in \R^n, \ \sum a_i|x_i| \leq \E S\}$ (Theorem 2 in \cite{NT-sc}), we obtain that for every $t \geq 1$,
\begin{equation}\label{eq:up-bd-small-t}
\p{S \geq t\E S} \leq \p{X_1 \geq ta} = e^{-at} \leq \left(\frac{23}{24}\right)^{t}.
\end{equation}
This provides an improvement of \eqref{eq:J} for small enough $t$ (of course the point of \eqref{eq:J} is that it is optimal for large $t$). The same can be said about the upper bound in \eqref{eq:gamma} for $\gamma \geq 1$ (in view of \eqref{eq:S>ES} and the results from \cite{NT-s} for gamma distributions with parameter $\gamma \geq 1$). Complimentary to such concentration bounds are small ball probability estimates and anti-concentration phenomena, typically treating however the regime of $t = O(1/\E S)$ (under our normalisation). We refer for instance to the comprehensive survey \cite{NV} of Nguyen and Vu, as well as the recent work \cite{LM} of Li and Madiman for further results and references. Specific reversals of \eqref{eq:up-bd-small-t} concerning the exponential measure can be found e.g. in \cite{ENT} (Corollary 15), \cite{PV} (Proposition 3.4), \cite{PV2} ((5.5) and Theorem 5.7).

\subsection{Heavy-tailed distributions}\label{sec:heavy-tailed}
Janson's as well as this paper's techniques strongly rely on Chernoff-type bounds involving exponential moments to establish the \emph{largest-weight} summand tail asymptotics from \eqref{eq:J} or \eqref{eq:sym-exp}. Interestingly, when the exponential moments do \emph{not} exist, i.e. for heavy-tailed distributions, under some natural additional assumptions (subexponential distributions), a different phenomenon occurs: in the simplest case of i.i.d. summands, we have 
\[
\p{X_1 + \dots + X_n > t} = (1+o(1))\p{\max_{i \leq n} X_i > t} \qquad \text{as } t \to \infty,
\]
often called the single big jump or catastrophe principle. We refer to the monograph \cite{FKZ} (Chapters 3.1 and 5.1), as well as the papers \cite{Chen} and \cite{FKZ2} for extensions including weighted sums and continuous time respectively.

\subsection{Theorem \ref{thm:sym-exp} in a more general framework}
A careful inspection of the proof of Theorem \ref{thm:sym-exp} shows that thanks to Theorem 2.57 from \cite{BDR} (or the simpler but weaker bound \eqref{eq:gamma}), the former can be extended to the case where the $X_i$ have the same distribution as $\sqrt{Y_i}G_i$ with the $Y_i$ being i.i.d. gamma random variables and the $G_i$ independent standard Gaussian. For simplicity, we have decided to present it for the symmetric exponentials.

\end{document}